\documentclass[a4paper, draft, 12pt]{amsproc}
\usepackage{amssymb}
\usepackage{amsmath,amssymb,ulem,color}

\usepackage[hyphens]{url} 
\usepackage[dvips]{graphicx} 
\usepackage{cmdtrack} 
\usepackage{mathrsfs}                       

\theoremstyle{plain}
 \newtheorem{theorem}{Theorem}[section]
 \newtheorem{prop}{Proposition}[section]
 \newtheorem{lem}{Lemma}[section]
 
\theoremstyle{Definition}
 \newtheorem{exm}{Example}[section]
 
 \newtheorem{dfn}{Definition}[section]
\theoremstyle{remark}

 \numberwithin{equation}{section}

\renewcommand{\leq}{\leqslant}
\renewcommand{\geq}{\geqslant}

\title[Invariant Properties of Fatou Set, Julia Set and Escaping Set ...]{Invariant Properties of Fatou Set, Julia Set and Escaping Set of Holomorphic Semigroup}



\author[B. H. Subedi]{\bfseries  Bishnu Hari Subedi}

\address{ %% Put here your affiliation; street address is not required
Central Department of Mathematics \\ % \hfill (Received 00 00 201?)\\
Institute of Science and Technology   \\ %\hfill (Revised  00 00 201?)\\
Tribhuvan University   \\ %\hfill (Revised  00 00 201?)\\
Kirtipur, Kathmandu\\
Nepal}
\email{subedi.abs@gmail.com / subedi\_bh@cdmathtu.edu.np }

\author[A. Singh]{Ajaya Singh}
\address{Central Department of Mathematics, Institute of Science and Technology, Tribhuvan University, Kirtipur, Kathmandu, Nepal }
\email{singh.ajaya1@gmail.com / singh\_a@cdmathtu.edu.np} 
\thanks{This research work of the first author is supported by PhD faculty fellowship from University Grants Commission, Nepal. } %% optional
\begin{document}

{\begin{flushleft}\baselineskip9pt\scriptsize
%PUBLICATIONS DE L'INSTITUT MATH\'EMATIQUE\newline
%Nouvelle s\'erie, tome ??(1??)) (201?), od--do \hfill DOI: \\
%MANUSCRIPT
\end{flushleft}}
\vspace{18mm} \setcounter{page}{1} \thispagestyle{empty}

\begin{abstract}
In this paper, we prove that escaping set of transcendental semigroup is S-forward invariant. We also prove that if holomorphic semigroup is abelian, then Fatou set, Julia set and escaping set are S-completely invariant. We see certain cases and conditions that the holomorphic semigroup dynamics exhibits same dynamical behavior just like the classical complex dynamics. Frequently, we also examine certain amount of connection and contrast between classical complex dynamics and holomorphic semigroup dynamics.  
\end{abstract}

\maketitle

\section{Introduction}
It is quite natural to extend the Fatou-Julia-Eremenko  theory of the iteration of single holomorphic map in complex plane $ \mathbb{C} $ or extended complex plane $ \mathbb{C}_{\infty} $ to composite of the family of holomorphic maps. Let $ \mathscr{F} $ be a space of holomorphic maps on $ \mathbb{C} $ or $ \mathbb{C}_{\infty} $. For any map $ \phi \in \mathscr{F} $,  $ \mathbb{C} $ or $ \mathbb{C}_{\infty} $ is naturally partitioned  into two subsets: the set of normality and its complement. We say that a family $ \mathscr{F} $ is normal if each sequence from the family has a subsequence which either converges uniformly on compact subsets of $ \mathbb{C} $ or $ \mathbb{C}_{\infty} $ or diverges uniformly to $ \infty $. The set of normality or Fatou set $ F(\phi) $ of the map $ \phi \in \mathscr{F} $ is the largest open set on which the iterates $ \phi^{n} = \phi \circ \phi \circ \ldots \circ \phi$ (n-fold composition of $ \phi $ with itself) is a normal family.  The complement $ J(\phi) $ is the Julia set. A maximally connected subset of Fatou set $ F(f) $ is a Fatou component.  
 
 Semigroup $ S $ is a very classical algebraic structure with binary composition that satisfies associative law. It naturally arose from the general mapping of a set into itself. So a set of holomorphic maps on $ \mathbb{C} $ or $ \mathbb{C}_{\infty} $ naturally forms a semigroup. Here, we take a set $ A $ of holomorphic maps and construct a semigroup $ S $ consisting of all elements that can be expressed as a finite composition of elements in $ A $. We say such a semigroup $ S $ by holomorphic semigroup generated by the set $ A $.  
For our simplicity, we denote the class of all rational maps of $ \mathbb{C_{\infty}} $ by $ \mathscr{R} $ and class of all transcendental entire maps of $ \mathbb{C} $ by $ \mathscr{E} $. 
Our particular interest is to study of dynamics of the families of above two classes of  holomorphic maps.   For a collection $\mathscr{F} = \{f_{\alpha}\}_{\alpha \in \Delta} $ of such maps, let 
$$
S =\langle f_{\alpha} \rangle
$$ 
be a \textit{holomorphic semigroup} generated by them. Here $ \mathscr{F} $ is either a collection $ \mathscr{R} $ of rational maps or a collection $ \mathscr{E} $ of transcendental entire maps.   $ \Delta $ represents an index set to which $ \alpha $  belongs is finite or infinite.  
Here, each $f \in S$ is a holomorphic function and $S$ is closed under functional composition. Thus $f \in S$ is constructed through the composition of finite number of functions $f_{\alpha_k},\;  (k=1, 2, 3,\ldots, m) $. That is, $f =f_{\alpha_1}\circ f_{\alpha_2}\circ f_{\alpha_3}\circ \cdots\circ f_{\alpha_m}$. In particular,  if $ f_{\alpha} \in \mathscr{R} $, we say $ S =\langle f_{\alpha} \rangle$ a \textit{rational semigroup} and if  $ f_{\alpha} \in \mathscr{E} $, we say $ S =\langle f_{\alpha} \rangle$ a \textit{transcendental semigroup}. 

A semigroup generated by finitely many holomorphic functions $f_{i}, (i = 1, 2, \ldots, \\ n) $  is called \textit{finitely generated  holomorphic semigroup}. We write $S= \langle f_{1},f_{2},\ldots,f_{n} \rangle$.
 If $S$ is generated by only one holomorphic function $f$, then $S$ is \textit{cyclic semigroup}. We write $S = \langle f\rangle$. In this case, each $g \in S$ can be written as $g = f^n$, where $f^n$ is the nth iterates of $f$ with itself. Note that in our study of  semigroup dynamics, we say $S = \langle f\rangle$  a \textit{trivial semigroup}. 
  
The following result will be clear from the definition of holomorphic semigroup. It shows that every element of holomorphic semigroup can be written as finite composition of the sequence of $f_{\alpha} $
\begin{prop}\label{ts1}
 Let $S   = \langle f_{\alpha} \rangle$ be an arbitrary  holomorphic semigroup. Then for every $ f \in S $,  $f^{m} $(for all $ m \in \mathbb{N}$) can be written as $f^{m} =f_{\alpha_1}\circ f_{\alpha_2}\circ f_{\alpha_3}\circ \cdots\circ f_{\alpha_p}$ for some $ p\in \mathbb{N} $.
 \end{prop}
Let  $ f $ be a holomorphic map. We say that  $ f $ \textit{iteratively divergent} at $ z \in \mathbb{C} $ if $  f^n(z)\rightarrow \alpha \; \textrm{as} \; n \rightarrow \infty$, where $ \alpha $  is an essential singularity of $ f $. A sequence $ (f_{k})_{k \in \mathbb{N}} $ of holomorphic maps is said to be \textit{iteratively divergent} at $ z $ if $ f_{k}^{n}(z) \to\alpha_{k} \;\ \text{as}\;\ n\to \infty$ for all $ k \in \mathbb{N} $, where $ \alpha_{k} $  is an essential singularity of $ f_{k} $ for each $ k $.  Semigroup $ S $ is \textit{iteratively divergent} at $ z $ if $f^n(z)\rightarrow \alpha_{f} \; \textrm{as} \; n \rightarrow \infty$, where $ \alpha_{f} $  is an essential singularity of each $ f \in S $. Otherwise, a function $ f  $, sequence $ (f_{k})_{k \in \mathbb{N}} $ and semigroup $ S $  are said to be iteratively bounded at $ z $.

Based on the definition of classical complex dynamics (that is,  on the Fatou-Julia-Eremenko theory of a complex analytic function), the Fatou set, Julia set and escaping set in the settings of holomorphic semigroup are defined as follows.
\begin{dfn}[\textbf{Fatou set, Julia set and escaping set}]\label{2ab} 
\textit{Fatou set} of the semigroup $S$ is defined by
  \[F (S) = \{z \in \mathbb{C}: S\;\ \textrm{is normal in a neighborhood of}\;\ z\}\] 
and the \textit{Julia set} $J(S) $ of $S$ is the compliment of $ F(S) $. If $ S $ is a transcendental semigroup, the \textit{escaping set} of $S$ is defined by 
$$
I(S)  = \{z \in \mathbb{C}: S \;  \text{is iteratively divergent at} \;z \}
$$
We call each point of the set $  I(S) $ by \textit{escaping point}.  Any maximally connected subset $ U $ of the Fatou set $ F(S) $ is called \textit{Fatou component}.       
\end{dfn} 
It is obvious that $F(S)$ is the largest open subset of $\mathbb{C}$ on which the family $\mathscr{F} $ in $S$ (or semigroup $ S $ itself) is normal. Hence its compliment $J(S)$ is a smallest closed set for any  semigroup $S$. Whereas the escaping set $ I(S) $ is neither an open nor a closed set (if it is non-empty) for any semigroup $S$.  
        
If $S = \langle f\rangle$, then $F(S), J(S)$ and $I(S)$ are respectively the Fatou set, Julia set and escaping set in classical complex dynamics. In this situation we simply write: $F(f), J(f)$ and $I(f)$. 

The fundamental contrast between classical complex dynamics and semigroup dynamics appears by different algebraic structure of corresponding semigroups. In fact, non-trivial semigroup (rational or transcendental) need not be, and most often will not be abelian. However, trivial semigroup is cyclic and therefore abelian. As we discussed before, classical complex dynamics is a dynamical study of trivial (cyclic)  semigroup whereas semigroup dynamics is a dynamical study of non-trivial semigroup. 

The following characterization of escaping set will be clear from the definition \ref{2ab} of escaping set and proposition \ref{ts1}, which can be an alternative definition.

\begin{theorem}\label{ad1}
If a complex number $ z \in \mathbb{C} $ is escaping point of any transcendental semigroup $ S $, then  every sequence in $ S $ has a subsequence which diverges to $ \infty $ at $ z $.
\end{theorem}
The following immediate relations hold for any $ f \in S $  from the definition \ref{2ab}. Indeed, it shows certain connection between classical complex dynamics and semigroup dynamics.

\begin{theorem}\label{1c}
Let $ S $  be a semigroup. Then
\begin{enumerate}  
 \item $F(S) \subset F(f)$ for all $f \in S$  and hence  $F(S)\subset \bigcap_{f\in S}F(f)$. 
\item $ J(f) \subset J(S) $ for all $f \in S$.
 \item  $I(S) \subset I(f)$ for all $f \in S$  and hence  $I(S)\subset \bigcap_{f\in S}I(f)$ in the case of transcendental semigroup $ S $.
\end{enumerate}
\end{theorem}

Hinkkanen and Martin proved the following results ({\cite[Lemma 3.1 and Corollary 3.1]{hin}}).
\begin{theorem}\label{perf}
 Let $ S $ be a rational semigroup. Then
 Julia set $ J(S) $ is perfect and $ J(S) = \overline{\bigcup_{f \in S} J(f)} $ 
 \end{theorem}

K. K. Poon proved the following results ({\cite[Theorems 4.1 and 4.2] {poo}}).
\begin{theorem}\label{perf1}
 Let $ S $ be a transcendental semigroup. Then
 Julia set $ J(S) $ is perfect and $ J(S) = \overline{\bigcup_{f \in S} J(f)} $ 
 \end{theorem}

From the  theorem \ref{1c} ((1) and (3)), we can say that the Fatou set and the escaping set may be empty.
For example, the escaping set of semigroup $ S = \langle f, g \rangle $ generated by functions $ f(z) = e^{z} $ and $ g(z) =  e^{-z}$ is empty (the particular function $ h = g \circ f^{k} \in S $ (say) is iteratively bounded at any $ z \in I(f) $). There are several transcendental semigroups where Fatou set and escaping set are non-empty (see for instance {\cite[Examples 3.2 and 3.3]{kum4}} and
 {\cite[Examples 2.6 and 2.7]{kum2}})

\section{Invariant features of Fatou set, Julia set and escaping set}

The main contrast between classical complex dynamics and semigroup dynamics will appear in the invariant features of Fatou set, Julia set and escaping set. Note that invariant feature  is considered a very basic and fundamental structure of  these sets. 

\begin{dfn}[\textbf{Forward, backward and completely invariant set}]
For a given semigroup $S$, a set $U\subset \mathbb{C}$ is said to be $ S $-\textit{forward invariant}\index{forward ! invariant set} if $f(U)\subset U$ for all $f\in S$. It is said to be $ S $-\textit{backward invariant}\index{backward ! invariant set} if $f^{-1}(U) = \{z \in \mathbb{C}: f(z) \in U \}\subset U$ for all $ f\in S $ and it is called $ S $-\textit{completely invariant}\index{completely invariant ! set} if it is both S-forward and S-backward invariant.
\end{dfn}
If $ S $ is a rational semigroup, then Hinkkanen and Martin {\cite [Theorem 2.1]{hin}} proved the following result.
\begin{theorem}\label{fi}
The Fatou set $ F(S) $ of $ S $ is S-forward invariant and Julia set $ J(S) $ of $ S $ is S-backward invariant.
\end{theorem}

If $ S $ is a transcendental semigroup, then K. K. Poon  {\cite[Theorem 2.1] {poo}} proved the following result.
\begin{theorem}\label{fi1}
The Fatou set $ F(S) $ of $ S $ is S-forward invariant and Julia set $ J(S) $ of $ S $ is S-backward invariant.
\end{theorem}
Dinesh Kumar and Sanjay Kumar {\cite[Theorem 4.1]{kum2}} proved the following result that shows  escaping set $I(S)$ is also S-forward invariant.
Here,  we provide another proof based on our definition \ref{2ab} of escaping set.
\begin{theorem}\label{fi2}
The escaping set $ I(S) $ of transcendental semigroup $ S $ is S-forward invariant.
\end{theorem}
Before proving this theorem \ref{fi2}, we define and discuss some special subsets of semigrioup $ S $.

\begin{dfn}[\textbf{Subsemigroup, left ideal, right ideal and ideal}] A non-empty subset $ T $ of  semigroup $ S $ is a subsemigroup of $ S $ if $ f \circ g \in T $ for all $ f, \; g \in T $.
A non-empty subset $ I $ of $ S $ is called a left (or right) ideal of $ S $ if $ f \circ g \in I $ (or $ g \circ f \in I )$ for all $ f \in S $ and $ g\in I $. $ I $ is called an ideal of $ S $ if it is both left and right ideal of $ S $. 
 \end{dfn}
 Every (left or right) ideal is a subsemigroup but converse may not hold. The left ideal,  right ideal  and an ideal of a semigroup $ S $ can be constructed easily as shown in the following proposition. 
 \begin{prop}\label{ss1}
 For any $ f \in S $, the set $ S \circ f = \{g \circ f: g \in S\} $ is a left ideal of $ S $, the set $ f \circ S = \{f \circ g: g \in S\} $  is an right ideal of $ S $ and the set $ S\circ f \circ S = \{h \circ f \circ g: h, g \in S \} $ is an ideal of $ S $. For any $ n \in \mathbb{N}, \;  (g \circ f)^{n}  = h \circ f $ for some $ h \in S $ and $ (h \circ f \circ g)^{n}  = p\circ f \circ q$ for some $ p, q \in S $. 
\end{prop}

\begin{proof}[Proof of the theorem \ref{fi}]
Let $ z\in I(S) $. Then  by the definition \ref{2ab},  semigroup $ S $ is iteratively divergent at $ z $ and  for any $ g \in S $,  the subsemigroup $ S\circ g = \{f \circ g : f \in S \} $ is also iteratively divergent at $ z $.  That is, $ (f \circ g)^{n}(z) \to \infty\; \text{as}\;  n \to \infty $ for all $ f \in S $. By the propositon \ref{ss1}, for all $ n \in \mathbb{N} $, we have $ (f \circ g)^{n}  = h_{n_{i}} \circ g $ for some $ h_{n_{i}} \in S $ where $ n_{i} $ depends on $ n $. Therefore, $ (f \circ g)^{n}(z) \to \infty $ as $ n \to \infty \Rightarrow (h_{n_{i}} \circ g)(z) = h_{n_{i}}(g(z)) \to \infty$. Since $ h_{n_{i}}\in S $ is transcendental entire function with $  h_{n_{i}}(g(z)) \to \infty $,  we must have $ h_{n_{i}} = p^{m_{i}} $ for some $ p \in S $ and $ m_{i} \in \mathbb{N} $. Thus, from $  (f \circ g)^{n}(z) \to \infty \; \text{as}\; n \to \infty $ we get  $p^{m_{i}}(g(z)) \to \infty  $ as $ m_{i} \to \infty $. This proves that $S $ diverges at $ g(z) $, so $ g(z) \in I(S) $ for all $ g \in S $.  Hence $ I(S) $ is S-forward invariant.
\end{proof}

In the case of rational semigroup, Hinkkanen and Martin {\cite[Example1]{hin}} provided the following example that show that Fatou set $ F(S) $ need not be backward invariant and Julia set $ J(S) $ need not be forward invariant.
\begin{exm}\label{efi}
For a rational semigroup $ S =\langle z^{2}, z^{2}/a \rangle $, where $ a \in \mathbb{C}, |a| >1 $, the Fatou set $ F(S) = \{z : |z| <1\; \text{or}\; |z| > |a| \}  $ is not S-backward invariant and Julia set $ J(S)  = \{z : 1 \leq  |z| \leq  |a| \} $ is not S-forward invariant.  
\end{exm}

Fatou \cite{fat} and Julia \cite{jul} independently proved that classical Fatou set and Julia set of rational function are completely invariant. Note that all three sets $ F(f), J(f) $ and $ I(f) $ of transcendental entire function are also completely invariant. This is also fundamental contrast between classical complex dynamics and semigroup dynamics.

We prove under certain conditions, Fatou set is S-backward invariant  and Julia set is S-forward invariant  of a rational semigroup $ S $. First we define the notion of abelian holomorphic semigroup.
\begin{dfn}
Let $ S =\langle f_{\alpha} \rangle $  be a holomorphic semigroup. We say the $ S $ is abelian if $ f_{\alpha}\circ f_{\beta} = f_{\beta}\circ f_{\alpha} $ for all generators $ f_{\alpha} $ and $ f_{\beta} $ of $ S $. 
\end{dfn}

\begin{theorem}\label{bc}
The Fatou set $ F(S) $ is S-backward invariant and Julia set $ J(S)  $ is S-forward invariant if   $ S $ is an abelian rational semigroup. 
\end{theorem}
\begin{proof}
We prove that if $ g(z) \in F(S) $, then $ z \in F(S) $ for all $ g\in S$. This follows that $ g^{-1}(F(S)) \subset F(S) $ for all $ g \in S $. Suppose, $ g(z) \in F(S) $.  Let $ U $ be a neighborhood of $ g(z)$ such that $ \overline{U}\subset F(S) $. Then there is a subsequence $ (f^{n_{j}})$ such that $ f^{n_{j}}(g(z))\to f(g(z)) $ uniformly on $ U $, where $ f $ is rational or constant $ \infty $. Since $ S $ is abelian, so we have $g(f^{n_{j}}(z))\to g(f(z)) $ uniformly on $ U $ as well. This shows that $g\circ f^{n_{j}}\to g\circ f $ uniformly on $ U $. This proves $ z \in F(S) $ for all $ g\in S$. 
\end{proof}

From the theorems \ref{fi} and \ref{bc}, we can say that Fatou set $ F(S) $ and Julia set $ J(S) $ are S-completely invariant if $ S $ is an abelian  rational semigroup. 
The following example of Hinkkanen and Martin {\cite[Example 2]{hin}} is best for  the above theorem \ref{bc}.
\begin{exm}
The semigroup $ S= \langle T_{n}(z): n =0,1,2, \ldots \rangle $ generated by Tchebyshev polynomials $ T_{n}(z) $ defined by $ T_{0}(z) =1, T_{1}(z) = z $ and $ T_{n +1}(z) = 2zT_{n}(z) - T_{n -1}(z)  $ is abelian.
Therefore, by above theorem \ref{bc}, Fatou set $ F(S)  $ is S-forward invariant and Julia set $ J(S) $ is S-backward invariant. 
\end{exm}

Under the same condition of rational semigroup, we prove that Fatou set is S-backward invariant  and Julia set is S-forward invariant  if $ S $ is a transcendental semigroup.
\begin{theorem}\label{bc1}
The Fatou set $ F(S) $ is S-backward invariant and Julia set $ J(S)  $ is S-forward invariant if   $ S $ is an abelian transcendental semigroup. 
\end{theorem}
\begin{proof}
The proof is similar to theorem \ref{bc}. 
\end{proof}

In {\cite [Theorem 2.1]{kum1}}, Dinesh Kumar and Sanjay Kumar provided the following condition for backward invariance of $ I(S) $. Here, we give another proof based on our definition \ref{2ab} of escaping set.
\begin{theorem}\label{1d}
The escaping set $ I(S) $ of transcendental semigroup $ S $ is S-backward invariant if  $ S $ is an abelian transcendental semigroup. 
\end{theorem}
\begin{proof}
We prove that if $ g(z) \in I(S) $, then $ z \in I(S) $ for all $ g\in S$. This follows that $ g^{-1}(I(S)) \subset I(S) $ for all $ g \in S $.  This will be proved if we are able to prove its contrapositive statement:  if $ z\notin I(S) $, then  $ g(z) \notin I(S) $ for all $ g \in S $. Let $ z\notin I(S) $ there is some $ f\in S $  which is iteratively bounded at $ z $. That is, $ f^{n}(z)\nrightarrow\infty  $ as $ n\to \infty $.  In this case,  there exists a sequence $ (f_{k})_{k \in \mathbb{N}} $ in $ S $ containing $ f $  which is  iteratively bounded at $ z $ and all subsequences of this sequence containing $ f $ are also iteratively bounded at $ z $. Now,  for any $ g\in S $,   $ (f_{k}\circ g)_{k \in \mathbb{N}} $ is a sequence in $ S $.  Since $ S $ is abelian and $ g $ is a transcendental entire function, so by the continuity of $ g $ at  $ z \in \mathbb{C} $, we can write that $(f_{k}\circ g)(z) = (g\circ f_{k})(z) $ for all $ k \in \mathbb{N} $. From which it follows that the sequence $ (f_{k}\circ g)_{k \in \mathbb{N}} $ is iteratively bounded at $ z $. So, all subsequence of this sequences are iteratively bounded at $ z $. From the fact  $(f_{k}\circ g)(z) = (g\circ f_{k})(z) $ for all $ k \in \mathbb{N} $, we can say that all subsequences of the sequence $ (f_{k}\circ g)_{k \in \mathbb{N}} $ are  iteratively bounded at $ g(z) $. That is, $ g(z)\notin I(S) $ for all $ g \in S $. 
 Therefore, $ g^{-1}(I(S)) \subset I(S) $ for all $ g \in S $. This proves that $ I(S) $ is backward invariant.
\end{proof}

From the result of above theorems \ref{bc1} and  \ref{1d}, we can conclude that Fatou set $ F(S) $, Julia set $ J(S) $ and escaping set $ I(S) $ are S-completely invariant if $ S $ is an abelian transcendental semigroup.  For example, the following semigroups 
\begin{enumerate}
\item $\langle z +\gamma \sin z, \; z +\gamma \sin z + 2k \pi \rangle$,  
\item $\langle z +\gamma \sin z, \; -z -\gamma \sin z + 2k \pi \rangle$,
\item $\langle z + \gamma e^{z}, \; z + \gamma e^{z} + 2k \pi i \rangle$, 
\item  $\langle z - \sin z, \; z - \sin z + 2\pi \rangle$, 
\item $\langle e^{\gamma z}, \; e^{\gamma z + \frac{2\pi i}{\gamma}}\rangle$, where $0<\gamma < e^{-1}$,  
\end{enumerate} 
are abelian transcendental  semigroups, so their Fatou set, Julia set and escaping sets are S-completely invariant. 

The theorems \ref{bc}, \ref{bc1} and \ref{1d} give a kind of connection between classical complex dynamics and semigroup dynamics. We got completely invariant  structure of  Fatou set, Julia set and escaping set in both classical complex dynamics and semigroup dynamics because of their  associated  abelian semigroups. This was our expectation that abelian semigroup (rational and transcendental) must exhibit same dynamical features as in classical complex dynamics.
Note that it will not better to  conclude that Fatou set, Julia set and escaping set can not be completely invariant unless the semigroup is abelian.

If $ S $ is a finitely generated rational semigroup, then Hiroki Sumi {\cite[Lemma 1.1.4 (2)]{sum}} proved the following results.
\begin{theorem}\label{fn2}
If $ S = \langle f_{1}, f_{2},\ldots, f_{n}\rangle $ is a finitely generated rational semigroup, then 
$ F(S) = \bigcap_{i=1}^{n} f_{i}^{-1}(F(S))$  and $ J(S) = \bigcup_{i=1}^{n} f_{i}^{-1}(J(S))$ 
\end{theorem}

In the case of finitely generated transcendental semigroup, we prove following result which is analogous to above theorem \ref{fn2}.
\begin{theorem}\label{fn3}
If $ S = \langle f_{1}, f_{2},\ldots, f_{n}\rangle $ is a finitely generated transcendental semigroup, then 
$ F(S) = \bigcap_{i=1}^{n} f_{i}^{-1}(F(S))$  and $ J(S) = \bigcup_{i=1}^{n} f_{i}^{-1}(J(S))$ 
\end{theorem}
\begin{proof}
The Fatou set $ F(S)  $ is S-forward invariant in general (theorem \ref{fi1}). So, $ f_{i}(F(S))\subset F(S) $ and it follows $ F(S) \subset \bigcap_{i =1}^{n} f^{-1}_{i}(F(S)) $ for all $ i $. 

Next, let $ z_{0} \in \bigcap_{i =1}^{n} f^{-1}_{i}(F(S)) $.  Then (say) $ w_{i}= f_{i}(z_{0}) \in F(S)$ for all $ i $. The semigroup $ S $ is normal at $ w_{i} $ for all $ i $. In other words, every $ g \in S $ is equicontinuous at $ w_{i} $ for all $ i $. So, for any $ \epsilon >0 $, there is $ \delta > 0  $ such that 
$$
d(g(w), g(w_{i})) < \epsilon, \; \text{whenever}\; \;  d(w, w_{i}) < \delta
$$
for all $ w \in F(S) $ and $ i =1,2,3,\ldots, n $. Where $ d $ represents Euclidean metric on $ \mathbb{C} $. For such $ \delta $, there is  $ \eta > 0$  such that 
$$
d(f_{i}(z), f_{i}(z_{0})) < \delta, \; \text{whenever}\; \;  d(z, z_{0}) < \eta
$$
for all  $ z \in \bigcap_{i =1}^{n} f^{-1}_{i}(F(S)) $ and $ i =1,2,3,\ldots, n $. Thus, ultimately, we conclude that 
$$
d(g(f_{i}(z)), g(f_{i}(z_{0})))<\epsilon \;\;  \text{whenever}\; \;  d(z, z_{0}) < \eta
$$
for all  $ z \in \bigcap_{i =1}^{n} f^{-1}_{i}(F(S)) $ and $ i =1,2,3,\ldots, n $. Here $ S = \bigcup_{i =1}^{n} (S \circ f_{i}) $. So $ S $ is equicontinuous at $ z_{0} $. That is, $ z_{0} \in F(S)$. Hence, $ F(S) = \bigcap_{i=1}^{n} f_{i}^{-1}(F(S))$. 
Second part of the theorem easily follows as
$$
J(S) = \mathbb{C} - F(S) = \mathbb{C} - \bigcap_{i=1}^{n} f_{i}^{-1}(F(S)) = \bigcup_{i =1}^{n}(\mathbb{C} -f_{i}^{-1}(F(S)) = \bigcup_{i =1}^{n}f_{i}^{-1}(J(S)).
$$
\end{proof}

The relation $ J(S) = \bigcup_{i=1}^{n} f_{i}^{-1}(J(S))$ for Julia set is called \textit{backward self similarity}. With this property, dynamics of semigroup can be regarded as \textit{backward iterated function systems}.

The theorems \ref{fi} and \ref{1d} can be used directly to express the escaping set $ I(S) $ as a finite intersection of all pre-images of itself under the generators of the semigroup  $ S $ if it is finitely generated.
The following result is due to Dinesh Kumar and Sanjay Kumar {\cite[Theorem 2.6]{kum1}}. Here, we give an alternative proof. 
\begin{theorem}\label{fn1}
If $ S = \langle f_{1}, f_{2},\ldots, f_{n}\rangle $ is a finitely generated transcendental semigroup, then 
$ I(S) = \bigcap_{i=1}^{n} f_{i}^{-1}(I(S))$ if $ S $ is abelian semigroup. 
\end{theorem}
\begin{proof}
$I(S)$ is forward invariant in general and under the assumption given in the statement,  it is also backward invariant (see for instance theorem \ref{1d}). So, we have $ f_{i}(I(S))\subset I(S) $ and $ f^{-1}_{i}(I(S))\subset I(S) $ for all $ 1\leq i \leq n $. From which we get respectively $ I(S) \subset \bigcap_{i =1}^{n} f^{-1}_{i}(I(S)) $ and $ I(S) \supset \bigcap_{i =1}^{n} f^{-1}_{i}(I(S)) $ for all $ 1\leq i \leq n $. Thus,  we get $ I(S) = \bigcap_{i=1}^{n} f_{i}^{-1}(I(S))$. 
\end{proof}

Later, we found that the condition of abelian in semigroup $ S $ is not required to hold essence of above theorem \ref{fn1}. For $ I(S) \subset \bigcap_{i =1}^{n} f^{-1}_{i}(I(S)) $ as above. Next, 
let $ z_{0} \in \bigcap_{i =1}^{n} f^{-1}_{i}(I(S)) $.  Then (say) $ w_{i}= f_{i}(z_{0}) \in I(S)$ for all $ i $. By the definition \ref{2ab}, semigroup $ S $ is iteratively divergent at $ w_{i} $ for all $ i $. By the theorem \ref{ad1}, every sequence $ (f_{j})_{j \in \mathbb{N}} $ in $ S $ has a subsequence $ (f_{j_{k}})_{k \in \mathbb{N}} $ which diverges to $ \infty $ at $ w_{i} = f_{i}(z_{0})  $. This proves that the sequence  $(f_{j_{k}} \circ f_{i})_{k, i \in \mathbb{N}} $ diverges to $ \infty $ at $ z_{0} $. Hence the result.

\section{Further results in abelian semigroup dynamics}
We have some further expectations that semigroup dynamics behave just like classical complex dynamics if  corresponding semigroup is abelian. We see more specific results in abelian semigroup dynamics. This  also exhibits a nice connection between classical complex dynamics and semigroup dynamics.
\begin{theorem}\label{abe}
Let $ S $ be an abelian rational  semigroup. Then $ J(S) = J(f) $ for all $ f \in S $ of degree at least two.
\end{theorem}
We recall the following result of Fatou \cite{fat1} and Julia \cite{jul1} concerning commuting rational functions.
\begin{lem} \label{per}
Let $ f $ and $ g $ be two rational functions of degree at least two such that $ f \circ g =g \circ f $. Then $ J(f) =J(g) $.
\end{lem}
\begin{proof}[Proof of the theorem  \ref{abe}]
Since semigroup $ S $ is abelian, so we have $ f_{i} \circ f_{j} = f_{j} \circ f_{i} $ for all generators $ f_{i} $ and $ f_{j} $ with $ i \neq j $. Then by above lemma \ref{per}, $ J(f_{i}) = J(f_{j}) $ for all $ i $ and $ j $ with $ i \neq j $. Also, every $ f \in S $ permutes with each generator $ f_{i} $ for all $ i $,  so again by the same lemma \ref{per}, $ J(f) = J(f_{i}) $ for all $ i $. This fact together with the fact of theorem \ref{perf}, we can conclude that $J(S) = J(f) $ for all $ f \in S $.
\end{proof}

The analogous result in transcendental semigroup may not hold in general because of the essence of above lemma \ref{per} is still unanswered for permutable transcendental entire functions. Julia sets for two permutable entire functions were studied in \cite{poo1, sing, tue} where we found certain conditions from which we can get the essence of above lemma \ref{per}. If we expose extra conditions in the statement, then result analogous to the theorem \ref{abe} holds in the case of transcendental semigroup. One of the analogous result of the theorem \ref{abe} was proved by K.K. Poon {\cite[Theorem 5.1]{poo}}. Here, we only give sketch of the proof similar to above  proof of the theorem \ref{abe}. 
\begin{theorem}\label{abe1}
Let $ S $ be an abelian transcendental semigroup in which each generator is of finite type. Then $ J(S) =J(f) $ for all $ f \in S $.
\end{theorem}
Before proving this theorem \ref{abe1}, We recall the notion of finite type, bounded type functions and other related objects.  Recall that the set  $CV(f) = \{w\in \mathbb{C}: w = f(z)\;\ \text{such that}\;\ f^{\prime}(z) = 0\} $ is called the set of \textit{critical values}.  The set 
$AV(f)$ consisting of all  $w\in \mathbb{C}$ such that there exists a curve  $\Gamma:[0, \infty) \to \mathbb{C}$ so that $\Gamma(t)\to\infty$ and $f(\Gamma(t))\to w$ as $t\to\infty$ is called the set of \textit{asymptotic values} of $ f $ and the set
$SV(f) =  \overline{(CV(f)\cup AV(f))}$
is called the set of \textit{singular values} of $ f $.  
If $SV(f)$ has only finitely many elements, then $f$ is said to be of \textit{finite type}. If $SV(f)$ is a bounded set, then $f$ is said to be of \textit{bounded type}.                                                                                                                             
The sets
$$\mathscr{S} = \{f:  f\;\  \textrm{is of finite type}\} 
\;\;  \text{and}\; \;                                                                                                                                                                                                                                                                                                                                                                                                                                                                                                                                                                                                                                                                                                                                                                                                                                                                                                                                                                                                                                                                                                                                                                                                                                                                                                                                                                                                                                                                                                                                                                                                                                                                                                                                                                                                                                                                                                                                                                                                                                                                                                                                                                                                                                                                                                                                                                                                                                                                                                                                                                                                                                                                                                                                                                                                                                                                                                                                                                                                                                                                                                                     
\mathscr{B} = \{f: f\;\  \textrm{is of bounded type}\}
$$
are respectively called \textit{Speiser class} and \textit{Eremenko-Lyubich class}.  
Again, we recall the following result of K. K. Poon {\cite[Lemmas 5.1 and 5.2]{poo}} concerning commuting transcendental maps.
\begin{lem}\label{per1}
Let $ f $ and $ g $ be two transcendental entire functions of finite type. Then $ f \circ g $ is of finite type. Moreover, if $ f $ and $ g $ are permutable, then $ J(f) =J(g) $.
\end{lem}
\begin{proof}[Proof of the theorem  \ref{abe1}]
Since semigroup $ S $ is abelian, then by above lemma \ref{per1}, $ J(f_{i}) = J(f_{j}) $ for all generators $f_ i $ and $ f_j $ with $ i \neq j $. Each $ f =f_{i_{1}}\circ f_{i_{2}} \circ \ldots f_{i_{m}} $ is of bounded type by the same lemma \ref{per1} and $ J(f) = J(f_{i}) $ for all $ i $. This fact together with the fact of theorem \ref{perf1}, we can conclude that $J(S) = J(f) $ for all $ f \in S $.. 
\end{proof}

We expect that the condition mentioned in the theorem \ref{abe1} will also be enough to hold $ I(S) = I(f) $  for all $ f \in S $.
\begin{theorem}\label{abet}
Let $ S $ is an abelian transcendental semigroup in which each generator is of finite type (or bounded type).  Then $ I(S) = I(f) $ for all $ f \in S $.
\end{theorem}

\begin{lem}\label{es21}
If $ f $ and $ g $ are permutable transcendental entire functions of finite type (or bounded type), then $ I(f) = I(g) $.
\end{lem}
\begin{proof}
As given in the statement of this lemma, Poon {\cite[Lemma 5.2]{poo}} showed that $ F(f) = F(g) $(Theorem \ref{abe1}). Eremenko and Lyubich \cite{ere1} proved that if transcendental function $ f\in \mathscr{B} $, then $ I(f)\subset J(f) $, and $ J(f) = \overline{I(f)}$.
For  any  function of finite type (or bounded type), we must have $ \overline{I(f)} = \overline{I(g)}$. This lemma will be proved if we show $ J(f) - I(f) = J(g) -I(g) $. Let $ z \in J(f) - I(f) $. Then $ z $ is a non-escaping point of $ J(f) $ and so the sequence $ (f^{n}) $  has a bounded subsequence at $ z $. $J(f) = J(g) $ implies that the sequence $ (g^{n}) $  has also a bounded subsequence at $ z $. So $ z \in J(g) -I(g) $. Therefore, $ J(f) - I(f) \subset J(g) -I(g) $. By similar fashion, we can show that $ J(g) -I(g) \subset J(f) - I(g) $. Hence, we got our claim. 

\end{proof}
\begin {proof}[Proof of the Theorem \ref{abet}]
Since semigroup $ S $ is abelian, so we have $ f_{i}\circ f_{j} = f_{j}\circ f_{i}$ for all generator $ f_{i} $ and $ f_{j} $ with $ i \neq j $. So by above lemma \ref{es21}, we have $ I(f_{i}) = I(f_{j}) $. Any $ f \in S $ can be written as $ f = f_{i_1}\circ f_{i_2}\circ f_{i_3}\circ \cdots\circ f_{i_m}$. By permutability of each $ f_{i} $, we can rearrange $ f_{i_{j}} $ and ultimately represented by 
$$
f = f_{1}^{t_{1}} \circ f_{2}^{t_{2}} \circ \ldots \circ f_{n}^{t_{n}}
$$
where each $ t_{k}\geq 0 $ is an integer for $ k = 1, 2, \ldots, n $. The lemms \ref{per1} can be applied repeatably to show each of $f_{1}^{t_{1}}, f_{2}^{t_{2}},\ldots, f_{n}^{t_{n}}  $ is of finite (or bounded) type and so $f = f_{1}^{t_{1}} \circ f_{2}^{t_{2}} \circ \ldots \circ f_{n}^{t_{n}}$ is itself finite  (or bounded) type. Since each $ f_{i} $  permutes with $ f $ and hence again by above lemma \ref{es21}, $ I(f_{i}) = I(f) $ for all $ f \in S $. Therefore, $ I(S) = I(f) $ for each $ f \in S $.
\end {proof}


\begin{thebibliography}{30}

\bibitem {ere1} Eremenko, A., and Lyubich, M.Y.: \textit{Dynamical properties of some classes of entire functions}, Ann. Inst. Fourier, Grenoble, 42 (1992), 989-1020.

\bibitem {fat} Fatou, P.: \textit{Sur les equations fonctionelles}, Bull. Soc. Math. France, 47 (1919), 161-271.

\bibitem {fat1} Fatou, P.: \textit{Sur les iterations analytique et les substitutions permutables}, J. Math. (9) 2 (1923), 343-384.

\bibitem {hin} Hinkkanen, A. and Martin, G.J.: \textit{The dynamics of semigroups of rational functions- I}, Proc. London Math. Soc. (3) 73, 358-384, (1996).

\bibitem {jul} Julia, G.: \textit{Memoire sur l' iteration des fractions rationelles}, J. Math. Pures Appl.  (8) 1 (1918), 47-245.

\bibitem {jul1} Julia, G.: \textit{Memoire sur la permutabilite des fractions rationalles}, Ann. Sci. Ecole Norm. Sup. 39 (3) (1922), 131-215.


\bibitem {kum2}Kumar, D. and Kumar, S.: \textit{The dynamics of semigroups of transcendental entire functions-II}, arXiv: 1401.0425 v3 (math.DS), May 22, 2014.

\bibitem {kum1} Kumar, D. and Kumar, S.: \textit{Escaping set and Julia set of transcendental semigroups}, arXiv:141.2747 v3 (math. DS) October 10, 2014. 


\bibitem {kum4}Kumar, D. and Kumar, S.: \textit{The dynamics of semigroups of transcendental entire functions-I}, Indian J. Pure  and Appli. Math. 46  (1), (2015), 11-24.

\bibitem {mor} Morosawa, S., Nishimura, Y., Taniguchi, M. and Ueda, T.: \textit{Holomorphic dynamics}, Cambridge University Press,Cambridge, UK, 2000.

\bibitem {poo} Poon, K.K.: \textit{Fatou-Julia theory on transcendental semigroups}, Bull. Austral. Math. Soc. Vol- 58(1998) PP 403-410.

\bibitem {poo1} Poon, K.K and Yang, C. C.: \textit{Dynamical behavior of two permutable entire functions}, Annales Polonici Mathematici, LXVIII 2 (1998), 159-163.

\bibitem {sing} Singh, A. P. and Wang, Y.: \textit{Julia sets of permutable holomorphic maps}, Sciences in China, Series A: Mathematics Vol. 49, No. 11 (2006), 1715-1721.

\bibitem {sum} sumi, H.: \textit{On dynamics of hyperbolic rational semigroup}, J. Math. Kyoto Univ. (JMKYAZ), 37 4 (1998), 717-733.

\bibitem {tue} Tuen, W. Ng.: \textit{Permutable entire functions and their Julia sets}, Math. Proc. Camb. Phil. Soc., 131 (2001), 129-138.



\end{thebibliography}
\end{document}